\documentclass[12pt,reqno]{article}

\usepackage{float}
\usepackage{color}
\usepackage{fullpage}
\usepackage{srcltx}
\usepackage{tikz}
\usepackage{psfig}
\usepackage{amsfonts}
\usepackage{latexsym}
\usepackage{epsf}
\usepackage{array}   
\usepackage{amssymb}
\usepackage{amsmath}
\usepackage{amsthm}
\usepackage{graphicx}
\usepackage{graphics}
\usepackage{algorithm}
\usepackage{url}
\usepackage{amsbsy}
\usepackage{algpseudocode}
\usepackage{algorithm}
\usepackage
[colorlinks=true,
linkcolor=green,
filecolor=brown,
citecolor=green, breaklinks=false]{hyperref}
\usepackage{algpseudocode}
\newcommand{\seqnum}[1]{\href{http://www.oeis.org/#1}{\underline{#1}}}

\usepackage[mathcal]{euler}
\usepackage{epsfig}

\title{Sequences Realized by the Subgroup Pattern of the Symmetric Group}
\author{Liam Naughton}
 
\newcommand{\Sym}{\mathsf{Sym}}

\usepackage{url}

\usepackage{amsthm}

\newtheorem{prop}{Proposition} 
\newtheorem{Lemma}{Lemma}
\theoremstyle{definition}
\newtheorem{defn}{Definition} 
\newtheorem{remark}{Remark}

\theoremstyle{definition}
\newtheorem{eg}{Example}
\begin{document}

\begin{center}
\vskip 1cm{\LARGE\bf Integer Sequences Realized by the Subgroup Pattern of the Symmetric Group
}
\vskip 1cm
\large
L. Naughton and G. Pfeiffer\\  
School of Mathematics, Applied Mathematics and Statistics\\  
National University of Ireland, Galway\\  
Ireland\\   
\href{mailto:liam.naughton@nuigalway.ie}{\tt } \\
\end{center}

\begin{abstract}
The subgroup pattern of a finite group $G$ is the table of marks of $G$ together with a list of representatives of the conjugacy classes of 
subgroups of $G$. In this article we describe a collection of sequences realized by the subgroup pattern of the symmetric group.
\end{abstract}

\section{Introduction}\label{sec:introduction}
The table of marks of a finite group $G$ was introduced by Burnside \cite{burn}. It is a matrix whose rows and columns are indexed by a list of representatives of 
the conjugacy classes of subgroups of $G$, where, for two subgroups $H,K \leq G$ the $(H,K)$ entry in the table of marks of $G$ is the number of fixed points of 
$K$ in the transitive action of $G$ on the cosets of $H$, ($\beta_{G/H}(K)$).  If $H_{1}, \dots, H_{r}$ is a list of representatives of the conjugacy classes of subgroups of $G$, the 
table of marks is then the $(r \times r)$-matrix 
\[ M(G) = ( \beta_{G/H_{i}}(H_{j}) )_{i,j = 1.,\dots,r}. \]
In much the same fashion as the character table of $G$ classifies matrix representations of $G$ up to isomorphism, the table of marks of $G$ 
classifies permutation representations of $G$ up to equivalence. It also encodes a wealth of information about the subgroup lattice of $G$ in a compact way. 
The \texttt{GAP} \cite{GAP4} library of tables of marks \texttt{Tomlib} \cite{tomlib} provides ready access to the tables of marks and conjugacy classes of subgroups 
of some $400$ groups. These tables have been produced using the methods described in \cite{pfe} and \cite{mcompaper}. The data exhibited in later sections has been computed using this library. The purpose of this article is to illustrate how interesting 
integer sequences related to the subgroup structure of the symmetric group $S_{n}$, and the alternating group $A_{n}$, can be computed from this data. This paper is organized as follows. In Section 
\ref{sec:counting subgroups} we study the conjugacy classes of subgroups of  $S_{n}$ for $n \leq 13$. In Section \ref{sec:Sequences from the Table of Marks} 
we examine the tables of marks of $S_{n}$ for $n \leq 13$ and describe how much more information regarding the subgroup structure of $S_{n}$ can be obtained. In 
Section \ref{sec:conn} we discuss the Euler Transform and its applications in counting subgroups of $S_{n}$.

\section{Counting Subgroups}\label{sec:counting subgroups}
Given a list of representatives $\{ H_{1}, \dots, H_{r} \}$ of $\mathrm{Sub}(G)/G$, the conjugacy classes of subgroups of $G$, we can enumerate those subgroups which satisfy particular properties. The numbers of 
conjugacy classes of subgroups of $S_{n}$ and $A_{n}$ are sequences \seqnum{A000638} and \seqnum{A029726} respectively in Sloane's encyclopedia \cite{intseq}. The \texttt{GAP} table of marks library \texttt{Tomlib}
 provides access to the conjugacy classes of subgroups of the symmetric and alternating groups for $n \leq 13$. Table~\ref{tab:sequences in Sn} records the number 
of conjugacy classes of subgroups of $S_{n}$ which are abelian, cyclic, nilpotent, solvable and supersolvable (SupSol). A similar table for the 
conjugacy classes of subgroups of the alternating groups can be found in \ref{appendix}.

  \begin{center}
\begin{table}[H]
  \arraycolsep12pt
  $\begin{array}{|r|r|r|r|r|r|r|} \hline 
& \text{\seqnum{A000638}} 
& \text{\seqnum{A218909}}
& \text{\seqnum{A000041}}
& \text{\seqnum{A218910}}
& \text{\seqnum{A218911}}
& \text{\seqnum{A218912}}
\\
\hline 
n& \text{$|\mathrm{Sub}(S_{n})/S_{n}|$} 
& \text{Abelian}
& \text{Cyclic}
& \text{Nilpotent}
& \text{Solvable}
& \llap{\text{SupSol}}
\\
\hline
 1  & 1          &     1 &       1  &   1       &       1       &       1\\           
 2  & 2          &     2 &       2  &   2       &       2       &       2\\           
 3  & 4          &     3 &       3  &   3       &       4       &       4\\          
 4  & 11         &     7 &       5  &   8       &       11      &       9\\         
 5  & 19         &     9 &       7  &   10      &       17      &       15\\        
 6  & 56         &    20 &      11  &   25      &       50      &       38\\       
 7  & 96         &    26 &      15  &   32      &       84      &       65\\      
 8  & 296        &    61 &      22  &   127     &       268     &       187\\     
 9  & 554        &    82 &      30  &   156     &       485     &       341\\    
 10 & 1593       &   180 &      42  &   531     &       1418    &       923\\  
 11 & 3094       &   236 &      56  &   648     &       2691    &       1789\\
 12 & 10723      &   594 &      77  &   3727    &       9725    &       6118\\
 13 & 20832      &   762 &     101  &   4221    &       18286   &       11616\\
\hline
  \end{array}$
  
\caption{Sequences in $S_{n}$}\label{tab:sequences in Sn}
\end{table}
\end{center}

\subsection{Subgroup Orders}

\noindent A question of historical interest concerns the orders of subgroups of $S_{n}$. In \cite{cameronsbook} Cameron writes:
\textit{
 The Grand Prix question of the Academie des Sciences, Paris, in 1860 asked 
``How many distinct values can a function of $n$ variables take?'' In other words what are the possible indices of subgroups of $S_{n}$.}
 For $n \leq 13$, Table \ref{sizes} records the numbers of different orders $\mathcal{O}(S_{n}),\mathcal{O}(A_{n}) $ of subgroups of $S_{n}$ and $A_{n}$.  
One might as well also enumerate the number of ``missing'' subgroup orders, that is, the number, $d(S_{n})$, of divisors $d$ such that 
$d \mid |S_{n}|$ but $S_{n}$ has no subgroup of order $d$. Table \ref{missingsizes} records the number of missing subgroup orders of $S_{n}$ and $A_{n}$ for $n \leq 13$.

\begin{minipage}[b]{0.4\linewidth}\centering
\begin{table}[H]
\begin{center}
\begin{tabular}{|l|r|r|}
\hline
 & \seqnum{A218913} & \seqnum{A218914} \\
\hline
$n$ & $\mathcal{O}(S_{n})$ & $\mathcal{O}(A_{n})$\\
\hline
1 & 1 & 1\\
2 & 2 & 1\\
3 & 4 & 2\\
4 & 8 & 5\\
5 & 13 & 9\\
6 & 21 & 15\\
7 & 31 & 22\\
8 & 49 & 38\\
9 & 74 & 59\\
10 & 113 & 89\\
11 & 139 & 115\\
12 & 216 & 180\\
13 & 268 & 226 \\
\hline
\end{tabular}
\caption{Subgroup Orders}
\label{sizes}
\end{center}
\end{table}
\end{minipage}
\hspace{0.2cm}
\begin{minipage}[b]{0.5\linewidth}
\centering
\begin{table}[H]
 \begin{center}
\begin{tabular}{|l|r|r|}
\hline
 & \seqnum{A218915} & \seqnum{A218916}\\
\hline
$n$ & $d(S_{n})$ & $d(A_{n})$\\
\hline
1 & 0 & 0\\
2 & 0 & 0\\
3 & 0 & 0\\
4 & 0 & 1\\
5 & 3 & 3\\
6 & 9 & 9\\
7 & 29 & 26\\
8 & 47 & 46\\
9 & 86 & 81\\
10 & 157 & 151\\
11 & 401 & 365\\
12 & 576 & 540\\
13 & 1316 & 1214\\
\hline
 \end{tabular}
\caption{Missing Subgroup Orders}
\label{missingsizes}
 \end{center}
\end{table}
\end{minipage}
\section{Counting Using the Table of Marks}\label{sec:Sequences from the Table of Marks}
If in addition to a list of conjugacy classes of subgroups of $G$, the table of marks of $G$ is also available, or can be computed, one can say quite a lot about 
the structure of the lattice of subgroups of $G$.
We begin this section by giving some basic information about tables of marks and then go on to describe how we can count incidences and edges in the lattice of 
subgroups.
\subsection{About Tables of Marks}
Let $G$ be a finite group and let $\mathrm{Sub}(G)$ denote the set of subgroups of $G$. By $\mathrm{Sub}(G)/G$ we denote the set of conjugacy classes of subgroups 
of $G$. For $H,K \in \mathrm{Sub}(G)$ let 
\[ \beta_{G/H}(K) = \# \{ Hg \in G/H : (Hg)k = Hg \mbox{ for all } k \in K \} = \# \{ g \in G : K \leq H^g \}/|H| \]
denote the mark of $K$ on $H$. This number depends only on the $G$-conjugacy classes of $H$ and $K$. Note that $\beta_{G/H}(K) \neq 0 \Rightarrow |K| \leq |H|$. 
If $H_{1}, \dots, H_{r}$ is a list of representatives of the conjugacy classes of subgroups of $G$,
 the 
table of marks of $G$ is then the  $(r \times r)$-matrix 
\[ M(G) = ( \beta_{G/H_{i}}(H_{j}) )_{i,j = 1,\dots,r}. \]

If the subgroups in the transversal are listed by increasing group order the table of marks is a lower triangular matrix.
The table of marks $M(S_{4})$ of the symmetric group $S_{4}$ is shown in Figure \ref{toms4}.
\begin{figure}[H]
\begin{center}
\begin{tabular}{l|ccccccccccc}
$S_{4}/1 $&$ 24 $&$  $&$  $&$  $&$  $&$  $&$  $&$  $&$  $&$  $&$  $\\ 
$S_{4}/2 $&$ 12 $&$ 4 $&$  $&$  $&$  $&$  $&$  $&$  $&$  $&$  $&$  $\\ 
$S_{4}/2 $&$ 12 $&$ . $&$ 2 $&$  $&$  $&$  $&$  $&$  $&$  $&$  $&$  $\\ 
$S_{4}/3 $&$ 8 $&$ . $&$ . $&$ 2 $&$  $&$  $&$  $&$  $&$  $&$  $&$  $\\ 
$S_{4}/2^2 $&$ 6 $&$ 6 $&$ . $&$ . $&$ 6 $&$  $&$  $&$  $&$  $&$  $&$  $\\ 
$S_{4}/2^2 $&$ 6 $&$ 2 $&$ 2 $&$ . $&$ . $&$ 2 $&$  $&$  $&$  $&$  $&$  $\\ 
$S_{4}/4 $&$ 6 $&$ 2 $&$ . $&$ . $&$ . $&$ . $&$ 2 $&$   $&$  $&$  $&$  $\\ 
$S_{4}/S_3 $&$ 4 $&$ . $&$ 2 $&$ 1 $&$ . $&$ . $&$ . $&$ 1 $&$   $&$  $&$  $\\ 
$S_{4}/D_8 $&$ 3 $&$ 3 $&$ 1 $&$ . $&$ 3 $&$ 1 $&$ 1 $&$ . $&$ 1 $&$   $&$  $\\ 
$S_{4}/A_4 $&$ 2 $&$ 2 $&$ . $&$ 2 $&$ 2 $&$ . $&$ . $&$ . $&$ . $&$ 2 $&$ $\\ 
$S_{4}/S_4 $&$ 1 $&$ 1 $&$ 1 $&$ 1 $&$ 1 $&$ 1 $&$ 1 $&$ 1 $&$ 1 $&$ 1 $&$ 1 $\\ 
\hline
 &$ 1 $&$ 2 $&$ 2 $&$ 3 $&$ 2^2 $&$ 2^2 $&$ 4 $&$ S_3 $&$ D_8 $&$ A_4 $&$ S_4$
\end{tabular} 
\end{center}
\caption{Table of Marks $M(S_{4})$}
\label{toms4}
\end{figure}
\noindent As a matrix, we can extract a variety of sequences from the table of marks, the most obvious of which is the sum of the entries. The sum of the entries 
of $M(S_{n})$ for $n \leq 13$ is shown in Figure \ref{fig:sum M(G)}. We can also sum the entries on the diagonal to obtain the sequences in Figure \ref{fig:sum diag}.
\bigskip
\begin{minipage}[b]{0.4\linewidth}\centering
\begin{table}[H]
\begin{center}
\begin{tabular}{|l|l|l|}
\hline
 & \seqnum{A218917} & \seqnum{A218918} \\
\hline
$n$ & $S_{n}$ & $A_{n}$\\
\hline
1 & 1 & 1\\
2 & 4 & 1\\
3 & 18 & 5\\
4 & 146 & 39\\
5 & 681 & 192\\
6 & 7518 & 1717\\
7 & 58633 & 13946\\
8 & 952826 & 243391\\
9 & 11168496 & 2693043\\
10 & 232255571 & 38343715\\
11 & 3476965896 & 545787051\\
12 & 108673489373 & 15787210045 \\
13 & 1951392769558 & 268796141406 \\
\hline
\end{tabular}
\end{center}
\caption{Sum of $M(G)$}
\label{fig:sum M(G)}
\end{table}
\end{minipage}
\hspace{0.5cm}
\begin{minipage}[b]{0.5\linewidth}
\centering
\begin{table}[H]
\begin{center}
\begin{tabular}{|l|l|l|}
\hline
 & \seqnum{A218919} & \seqnum{A218920} \\
\hline
$n$ & $S_{n}$ & $A_{n}$\\
\hline
1 & 1 & 1\\
2 & 3 & 1\\
3 & 10 & 4\\
4 & 47 & 19\\
5 & 165 & 73\\
6 & 950 & 412\\
7 & 5632 & 2660\\
8 & 43772 & 21449\\
9 & 376586 & 184541\\
10 & 3717663 & 1827841\\
11 & 40555909 & 20043736\\
12 & 484838080 & 240206213\\
13 & 6286289685 & 3119816216 \\
\hline
\end{tabular}
\end{center}
\caption{Sum of the Diagonal}
\label{fig:sum diag}
\end{table}
\end{minipage}

\noindent We will now collect some elementary properties of tables of marks in Lemma \ref{lem:propertiesoftoms}. 
\begin{Lemma}\label{lem:propertiesoftoms}
\noindent Let $H, K \leq G$. Then the following hold:
\begin{enumerate}
\item[(i)] The first entry of every row of $M(G)$ is the index of the corresponding subgroup,
\[\beta_{G/H} (1) = |G:H|.\]
\item[(ii)] The entry on the diagonal is,
\[ \beta_{G/H}(H) = |N_G (H):H|. \]
\item[(iii)] The length of the conjugacy class $[H]$ of $H$ is given by,
\[ |[H]| = |G:N_G (H)| = \frac{\beta_{G/H} (1) }{\beta_{G/H} (H) }. \]
\item[(iv)] The number of conjugates of $H$ which contain $K$ is given by,
\[ |\lbrace H^a : a \in G, K \leq H^a \rbrace | = \frac{\beta_{G/H} (K)}{\beta_{G/H} (H) }.\]
\end{enumerate}
\end{Lemma}
\noindent The following formula which follows trivially from Lemma \ref{lem:propertiesoftoms} (iv) relates marks to incidences in the subgroup lattice of $G$.
\begin{align}\label{marksasincidence}
 \beta_{G/H}(K) = |N_{G}(H):H| \cdot \# \{H^g : K \leq H^g, g \in G \}.
\end{align}

\noindent As a first application of Formula \ref{marksasincidence} we obtain the following lemma which enables us to count the total number of subgroups of $G$.
\begin{Lemma}\label{lem:totalsub}
Given a list $\{ H_{1}, \dots, H_{r} \}$ of representatives of the conjugacy classes of subgroups of $G$, the total number of subgroups of $G$ is 
\[ |\mathrm{Sub}(G)| = \sum \limits_{i = 1}^{r} \frac{\beta_{G/H_{i}}(1)}{\beta_{G/H_{i}}(H_{i}) }. \]
\end{Lemma}
\begin{proof}
It follows from Formula \ref{marksasincidence} that for any subgroup $H \leq G$, $\frac{\beta_{G/H}(1)}{\beta_{G/H}(H)}$ is the length of the conjugacy class of 
$H$ in $G$.
\end{proof}
\noindent Table \ref{totalnumberofsubssymalt} lists the total number of subgroups of $S_{n}$ and $A_{n}$ for $n \leq 13$.

\begin{table}[H]

\begin{center}
\begin{tabular}{|l|l|l|}
\hline
&\seqnum{A029725} & \seqnum{A005432} \\
\hline
$n$ & $A_{n}$ & $S_{n}$\\
\hline
1 & 1 & 1\\
2 & 1 & 2\\
3 & 2 & 6\\
4 & 10 & 30\\
5 & 59 & 156\\
6 & 501 & 1455\\
7 & 3786 & 11300\\
8 & 48337 & 151221\\
9 & 508402 & 1694723\\
10 & 6469142 & 29594446\\
11 & 81711572 & 404126228\\
12 & 2019160542 & 10594925360\\
13 & 31945830446 & 175238308453\\
\hline
\end{tabular}
\end{center}
\caption{Total Number of Subgroups of $A_{n}$ and $S_{n}$}
\label{totalnumberofsubssymalt}
\end{table}

\subsection{Counting Incidences}
Another immediate consequence of Formula \ref{marksasincidence} is that by dividing each row of the table of marks of $G$ by its diagonal entry $\beta_{G/H}(H)$  we obtain a matrix 
 $\mathcal{C}(G)$ describing 
containments in the subgroup lattice of $G$, where the $(H,K)$-entry is 
\begin{align}
\mathcal{C}(H,K) = \# \{ K^{g} : H \leq K^{g}, g \in G \}.
\end{align}
Figure \ref{fig:subgroup containments} illustrates the containment matrix of the symmetric group $S_{4}$. 
\begin{center}
\begin{figure}[H]
\[ 
\begin{array}{c|ccccccccccc}
1 & 1 &  &  &  &  &  &  &  &  &  &  \\ 
2 & 3 & 1 &  &  &  &  &  &  &  &  &  \\ 
2 & 6 & . & 1 &  &  &  &  &  &  &  &  \\ 
3 & 4 & . & . & 1 &  &  &  &  &  &  &  \\ 
2^2 & 1 & 1 & . & . & 1 &  &  &  &  &  &  \\ 
2^2 & 3 & 1 & 1 & . & . & 1 &  &  &  &  &  \\ 
4 & 3 & 1 & . & . & . & . & 1 &   &  &  &  \\ 
S_3 & 4 & . & 2 & 1 & . & . & . & 1 &   &  &  \\ 
D_8 & 3 & 3 & 1 & . & 3 & 1 & 1 & . & 1 &   &  \\ 
A_4 & 1 & 1 & . & 1 & 1 & . & . & . & . & 1 & \\ 
S_4 & 1 & 1 & 1 & 1 & 1 & 1 & 1 & 1 & 1 & 1 & 1 \\ 
\hline
 & 1 & 2 & 2 & 3 & 2^2 & 2^2 & 4 & S_3 & D_8 & A_4 & S_4
\end{array}  \]
\caption{Containment Matrix : $\mathcal{C}(S_{4})$}
\label{fig:subgroup containments}
\end{figure}
\end{center}
The conjugacy classes of subgroups of $G$ are partially ordered by $[H] \leq [K]$ if $H \leq K^g$ for some $g \in G$ i.e. if $\mathcal{C}(H, K) \neq 0$.
Therefore we can easily obtain the incidence matrix, $\mathcal{I}(G)$, of the poset of conjugacy classes of subgroups 
of $G$ by replacing each nonzero entry in $\mathcal{C}(G)$, (or $M(G)$ ) by an entry $1$.  
Figure \ref{fig:incidence matrix S_(4)} shows the incidence matrix $\mathcal{I}(S_{4})$ of the poset of conjugacy classes of subgroups of $S_{4}$.
\begin{center}
\begin{figure}[H]
\[ 
\begin{array}{c|ccccccccccc}
1 & 1 &  &  &  &  &  &  &  &  &  &  \\ 
2 & 1 & 1 &  &  &  &  &  &  &  &  &  \\ 
2 & 1 & . & 1 &  &  &  &  &  &  &  &  \\ 
3 & 1 & . & . & 1 &  &  &  &  &  &  &  \\ 
2^2 & 1 & 1 & . & . & 1 &  &  &  &  &  &  \\ 
2^2 & 1 & 1 & 1 & . & . & 1 &  &  &  &  &  \\ 
4 & 1 & 1 & . & . & . & . & 1 &   &  &  &  \\ 
S_3 & 1 & . & 1 & 1 & . & . & . & 1 &   &  &  \\ 
D_8 & 1 & 1 & 1 & . & 1 & 1 & 1 & . & 1 &   &  \\ 
A_4 & 1 & 1 & . & 1 & 1 & . & . & . & . & 1 & \\ 
S_4 & 1 & 1 & 1 & 1 & 1 & 1 & 1 & 1 & 1 & 1 & 1 \\ 
\hline
 & 1 & 2 & 2 & 3 & 2^2 & 2^2 & 4 & S_3 & D_8 & A_4 & S_4
\end{array}  \]
\caption{Incidence Matrix : $\mathcal{I}(S_{4})$}
\label{fig:incidence matrix S_(4)}
\end{figure}
\end{center}
For comparison with Figure \ref{fig:incidence matrix S_(4)} we illustrate the 
poset of conjugacy classes of subgroups of $S_{4}$ in Figure \ref{lattices4}.
 
 \begin{figure}[H]
\begin{center}
\scalebox{0.8}{
\begin{tikzpicture}

 \node (1) at (10,0)[black]   {$1$};
\node (b2) at (8,2)[black]   {$2$};

\node (r2) at (6,2)[black]  {$2$};
\node (b3) at (12, 3)[black]   {$3$};
\node (b22) at (8, 4)[black]   {$2^2$};
\node (r22) at (4, 4)[black]  {$2^2$};
\node (r4) at (6, 4)[black]  {$4$};
\node (s3) at (13, 5)[black]  {$S_{3}$};
\node (d8) at (6, 6) [black] {$D_{8}$};
\node (a4) at (10, 8) [black] {$A_{4}$};
\node (s4) at (10, 10)[black]  {$S_{4}$};

\path
(1) edge[thick] (b2)
(1) edge[thick]  (b3)
(b2) edge[thick]  (b22)
(b3) edge[thick]  (a4)
(b22) edge[thick]  (a4)
(1) edge[thick] (r2)
(b2) edge[thick] (r22)
(b2) edge[thick] (r4)
(b22) edge[thick] (d8)
(a4) edge[thick] (s4)
(b3) edge[thick] (s3)
(r2) edge[thick] (s3)
(r2) edge[thick]  (r22)
(r4) edge[thick]  (d8)
(r22) edge[thick]  (d8)
(d8) edge[thick]  (s4)
(s3) edge[thick]  (s4);

\end{tikzpicture}}
\end{center}
\caption{Poset of Conjugacy Classes of Subgroups of $S_{4}$}
\label{lattices4}
 \end{figure}
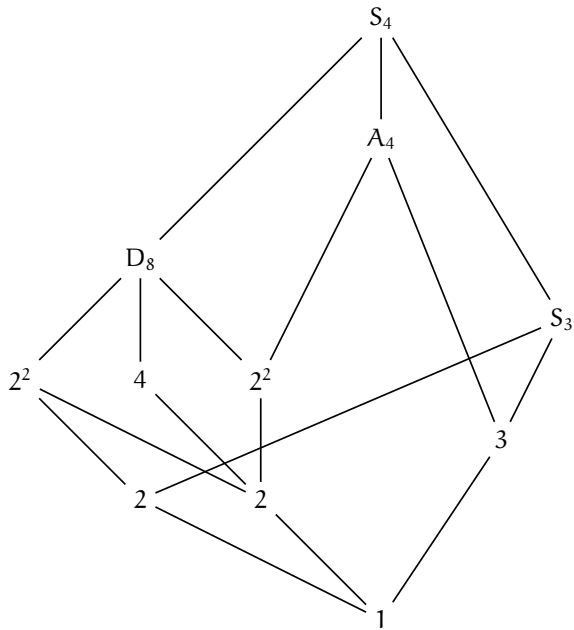
\begin{Lemma}
 The number of incidences in the poset of conjugacy classes of subgroups of $G$ is given by 
\[  \sum\mathcal{I}(G). \]
\end{Lemma} 
\begin{proof}
 The incidence matrix $\mathcal{I}(G)$ is obtained by replacing every nonzero entry in the table of marks by an entry $1$. By Formula \ref{marksasincidence} 
$\mathcal{I}(H,K) = 1$ if and only if $K$ is subconjugate to $H$ in $G$, i.e. if and only if $H$ and $K$ are incident in the poset of conjugacy classes of subgroups 
of $G$.
\end{proof}
Table \ref{fig:incidences in poset} lists the number of incidences in the poset of conjugacy classes of subgroups of $A_{n}$ and $S_{n}$ for $n \leq 13$.
\begin{Lemma}
 The total number of incidences in the entire subgroup lattice of $G$ is given by 
\[ \sum \mathcal{C}(G). \]
\end{Lemma}
\begin{proof}
 For $H,K \in \mathrm{Sub}(G)/G$ the $H,K$ entry in $\mathcal{C}(G)$ is the number of incidences between $H,K$ in the subgroup lattice of $G$. Thus summing 
over the entries in $\mathcal{C}(G)$ yields the total number of incidences in the entire subgroup lattice of $G$.
\end{proof}
\noindent Table \ref{fig : incidences in subgroup lattice} records the number of incidences in the subgroup lattices of $S_{n}$ and $A_{n}$ for $n \leq 13$.

\begin{table}[H]
\begin{minipage}[t]{0.4\linewidth}\centering
\begin{table}[H]
\begin{center}
\begin{tabular}{|l|l|l|}
\hline
 & \seqnum{A218921} & \seqnum{A218922} \\
\hline
$n$ & $S_{n}$  &$A_{n}$\\ \hline
1 & 1 & 1\\
2 & 3 & 1\\
3 & 9 & 3\\
4 & 44 & 13\\
5 & 101 & 32\\
6 &  523 & 128 \\
7 & 1195  & 330\\
8 & 6751  & 2309 \\
9 & 16986  & 4271 \\
10 & 87884  & 12468 \\
11 & 248635  & 33329 \\
12 & 1709781  & 196182 \\
13 & 4665651  & 490137  \\
\hline
\end{tabular}
\end{center}
\caption{Incidences in Poset}
\label{fig:incidences in poset}
\end{table}
\end{minipage}
\hspace{0.5cm}
\begin{minipage}[t]{0.5\linewidth}
\centering
\begin{table}[H]
\begin{center}
\begin{tabular}{|l|l|l|}
\hline
 & \seqnum{A218924} & \seqnum{A218923} \\
\hline
$n$ & $A_{n}$ & $S_{n}$\\ 
\hline
1 & 1 & 1\\
2 & 1 & 3\\
3 & 3 & 11\\
4 & 18 & 68\\
5 & 85 & 262\\
6 & 657 & 2261\\
7 & 4374 & 14032\\
8 & 55711 & 176245\\
9 & 530502 & 1821103\\
10 & 6603007 & 30883491\\
11 & 82736601 & 415843982\\
12 & 2032940127 & 10779423937\\
13 & 32102236563 & 177718085432\\
\hline
\end{tabular}
\end{center}
\caption{Incidences in Subgroup Lattice}
\label{fig : incidences in subgroup lattice}
\end{table}
\end{minipage}
\end{table}

\subsection{Counting Edges in Hasse Diagrams}
The table of marks also allows us to count the number of edges in both the Hasse diagrams of the poset of conjugacy classes of subgroups and the subgroup lattice of $G$. Computing 
such data requires careful analysis of maximal subgroups in the subgroup lattice.

Formula \ref{marksasincidence} describes containments in the poset of conjugacy classes of subgroups looking upward through the subgroup lattice of $G$. But we can 
also view marks as containments looking downward through the subgroup lattice of $G$.
\begin{Lemma}\label{lemma:edgesupward}
 Let $H, K \in \mathrm{Sub}(G)/G$. Then the number of conjugates of $H$ contained in $K$ is given by 
\[  E^{\uparrow} (H,K) = |\{ H^{g}, g \in G : H^{g} \leq K \}| = \frac{\beta_{G/K}(H) \beta_{G/H}(1)}{\beta_{G/H}(H) \beta_{G/K}(1)} \]
\end{Lemma}
\begin{proof}
 The total number of edges between the classes $[H]_{G}$ and $[K]_{G}$ can be counted in two different ways, as the length of the class times the 
number of edges leaving one member of the class. Thus
\[ |[H_{G}| \cdot |\{ H^{g} , g \in G : H^{g} \leq K \}| = |[K]_{G}| \cdot |\{ K^{g} , g \in G : K^{g} \geq H \}|. \]
By Formula \ref{marksasincidence} $|[H]_{G}| = \frac{\beta_{G/H}(1)}{\beta_{G/H}(H)}$ and $|[K]_{G}| = \frac{\beta_{G/K}(1)}{\beta_{G/K}(K)}$. 
Thus $E^{\uparrow}(H,K)$ can be expressed in terms of marks.
\end{proof}

\subsubsection{Identifying Maximal Subgroups}
It will be necessary, for the sections that follow, to identify for $H_{i} \in \mathrm{Sub}(G)/G$ which classes $H_{j} \in \mathrm{Sub}(G)/G$ are maximal in $H_{i}$. 
\begin{Lemma}\label{lem:max}
Let $H_{i} \in \mathrm{Sub}(G)/G = \{ H_{1}, \ldots, H_{r} \}$. Denote by $\rho_{i} = \{ j : H_{j} <_{G} H_{i} \}$ the set of indices in $\{1, \ldots , r \}$ of proper subgroups of $H_{i}$ up to 
conjugacy in $G$. Then the positions of all maximal subgroups of $H_{i}$ are given by
\begin{align}\label{form:max}
 \mathrm{Max}(H_{i}) = \rho_{i} \setminus \bigcup_{j \in \rho_{i}} \rho_{j} 
\end{align}
\end{Lemma}
The set of values $\rho_{i}$ are easily read off the table of marks of $G$ by simply identifying the nonzero entries in the row corresponding to $G/H_{i}$. 
Formula \ref{form:max} is implemented in \texttt{GAP} via the function \texttt{MaximalSubgroupsTom}.


\begin{Lemma}
 Let $\mathrm{Sub}(G)/G = \{ H_{1}, \ldots, H_{r} \}$ be a list of representatives of the conjugacy classes of subgroups of $G$. The number of edges in the Hasse diagram of the poset of conjugacy classes of subgroups of $G$ is given by 
\[ |E(\mathrm{Sub}(G)/G)| = \sum \limits_{i = 1}^{r} | \mathrm{Max}(H_{i}) |. \]
\end{Lemma}
\begin{proof}
 By Lemma \ref{lem:max}, $\mathrm{Max}(H_{i})$ is a list of the positions of the maximal subgroups of $H_{i}$ up to conjugacy in $G$. In the Hasse diagram of the poset 
$\mathrm{Sub}(G)/G$ each edge corresponds to a maximal subgroup.
\end{proof}
 Table \ref{fig:edges in poset} 
records the number of edges in the Hasse diagram of the poset of conjugacy classes of subgroups of $S_{n}$ and $A_{n}$ for $n \leq 13$.
In order to count the number of edges in the Hasse diagram of the entire subgroup lattice of $G$ we appeal to Formula \ref{marksasincidence} and Lemma \ref{lemma:edgesupward}.
\begin{Lemma}
 Let $\mathrm{Sub}(G)/G = \{ H_{1}, \ldots, H_{r} \}$ be as above. The total number of edges $E(L(G))$ in the Hasse diagram of the subgroup lattice of $G$ is given by
\[ E(L(G)) = \sum \limits_{i=1}^{r} \sum \limits_{j \in \mathrm{Max}(H_{i})} E^{\uparrow}(H_{i},H_{j}). \]
\end{Lemma}
\begin{proof}
 By restricting $E^{\uparrow}(H_{i},H_{j})$ to those classes $H_{i}, H_{j}$ which are maximal we obtain the number of edges connecting maximal subgroups of $G$. 
\end{proof}

Table \ref{fig:edges subgroup lattice} records the total number of edges in the Hasse diagram of the subgroup lattice of $S_{n}$ and $A_{n}$ for $n \leq 13$.

\begin{table}[H]
\begin{minipage}[]{0.4\linewidth}\centering
\begin{table}[H]
\begin{center}
\begin{tabular}{|l|l|l|}
\hline
 & \seqnum{A218925} & \seqnum{A218926} \\
\hline
$n$ & $S_{n}$& $A_{n}$ \\ \hline
1 & 0 & 0\\
2 & 1  &0 \\
3 &4 &1 \\
4 &17  &5 \\
5 &37  &13 \\
6 &149  &44 \\
7 &290  &98 \\
8 &1080  &419 \\
9 &2267  &722 \\
10 &8023 &1592 \\
11 &17249  &3304 \\
12 &72390  &12645\\
13 &153419  &24792 \\
\hline
\end{tabular}
\caption{Edges in Poset}
\label{fig:edges in poset}
\end{center}
\end{table}
\end{minipage}
\hspace{0.5cm}
\begin{minipage}[]{0.5\linewidth}
\centering
\begin{table}[H]
\begin{center}
\begin{tabular}{|l|l|l|}
\hline
 & \seqnum{A218928} & \seqnum{A218927} \\
\hline
$n$ & $A_{n}$ & $S_{n}$\\ \hline
1 & 0 & 0\\
2 & 0 & 1\\
3 & 1 & 8\\
4 & 15 & 66\\
5 & 168 & 501\\
6 & 2051 & 6469\\
7 & 19305 & 60428\\
8 & 283258 & 926743\\
9 & 3255913 & 11902600\\
10 & 46464854 & 240066343\\
11 & 670282962 & 3677270225\\
12 & 18723796793 & 108748156239\\
13 & 321480817412 & 1980478458627\\
\hline
\end{tabular}
\caption{Edges in Subgroup Lattice}
\label{fig:edges subgroup lattice}
\end{center}
\end{table}
\end{minipage}
\end{table}

\subsection{Maximal Property-P Subgroups}
For any property $P$ which is inherited by subgroups of $G$ we can use the table of marks of $G$ to enumerate the maximal property $P$ subgroups of $G$.  
\begin{Lemma}\label{lem:maxP}
Let $\mathrm{Sub}(G)/G = \{ H_{1}, \ldots, H_{r} \}$ and let $\rho = \{ i \in [1, \ldots, r ] : H_{i} \mbox{ is a property $P$ subgroup} \}$.
 Then the positions of the maximal property $P$ 
subgroups of $G$ are given by 

 \begin{align}\label{form:maxP}
 P(G) = \rho \setminus \bigcup_{j \in \rho}  \mathrm{Max}(H_{j})
\end{align}
\end{Lemma}

\noindent In Figure \ref{fig:max abelian subs s4} the classes of maximal abelian subgroups of $S_{4}$ are boxed.

 \begin{figure}[H]
\begin{center}
\scalebox{0.8}{
\begin{tikzpicture}

 \node (1) at (10,0)[black]   {$1$};
\node (b2) at (8,2)[black]   {$2$};

\node (r2) at (6,2)[black]  {$2$};
\node (b3) at (12, 3)[draw,blue]   {$\mathbf{3}$};
\node (b22) at (8, 4)[draw,blue]   {$\mathbf{2^2}$};
\node (r22) at (4, 4)[draw,blue]  {$\mathbf{2^2}$};
\node (r4) at (6, 4)[draw,blue]  {$\mathbf{4}$};
\node (s3) at (13, 5)[black]  {$S_{3}$};
\node (d8) at (6, 6) [black] {$D_{8}$};
\node (a4) at (10, 8) [black] {$A_{4}$};
\node (s4) at (10, 10)[black]  {$S_{4}$};

\path
(1) edge[thick,dashed,red] (b2)
(1) edge[thick,dashed,red]  (b3)
(b2) edge[thick,dashed,red]  (b22)
(b3) edge[thick]  (a4)
(b22) edge[thick]  (a4)
(1) edge[thick,dashed,red] (r2)
(b2) edge[thick,dashed,red] (r22)
(b2) edge[thick,dashed,red] (r4)
(b22) edge[thick] (d8)
(a4) edge[thick] (s4)
(b3) edge[thick] (s3)
(r2) edge[thick] (s3)
(r2) edge[thick,dashed,red]  (r22)
(r4) edge[thick]  (d8)
(r22) edge[thick]  (d8)
(d8) edge[thick]  (s4)
(s3) edge[thick]  (s4);

\end{tikzpicture}}
\end{center}
\caption{Maximal Abelian Subgroups of $S_{4}$}
\label{fig:max abelian subs s4}
 \end{figure}
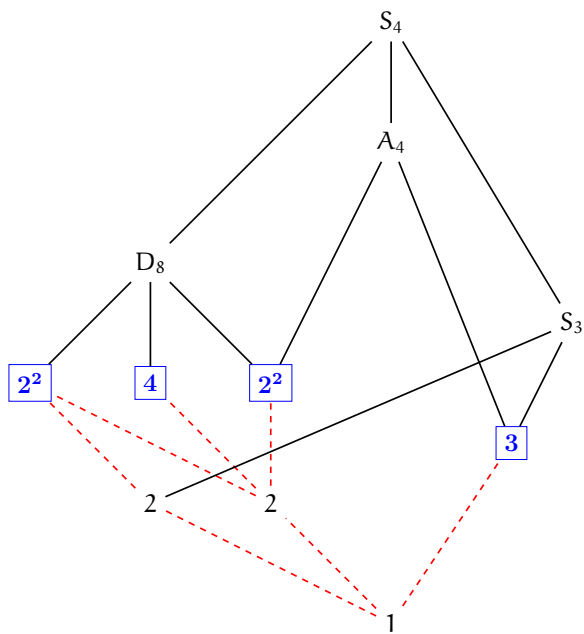

Table \ref{fig:max sn} records, for each of the properties listed across the first row of the table, the numbers of maximal property $P$ classes of subgroups of $S_{n}$.
 A similar table for the alternating groups can be found in the Appendix.

\begin{table}[H]
 \begin{center}
\begin{tabular}{|l|r|r|r|r|r|}
\hline
 & \seqnum{A218929} & \seqnum{A218930}& \seqnum{A218931} & \seqnum{A218932}& \seqnum{A218933} \\
\hline 
n & Solvable & SupSol & Abelian & Cyclic & Nilpotent\\
\hline
1 & 1 & 1 & 1 & 1 & 1\\
2 & 1 & 1 & 1 & 1 & 1\\
3 & 1 & 1 & 2 & 2 & 2\\
4 & 1 & 2 & 4 & 3 & 2\\
5 & 3 & 3 & 5 & 3 & 3\\
6 & 4 & 4 & 7 & 5 & 5\\
7 & 5 & 5 & 10 & 6 & 6\\
8 & 6 & 6 & 17 & 11 & 7\\
9 & 9 & 8 & 23 & 15 & 9\\
10 & 12 & 11 & 30 & 20 & 12\\
11 & 14 & 14 & 41 & 24 & 15\\
12 & 17 & 19 & 61 & 34 & 20\\
13 & 24 & 23 & 80 & 43 & 25\\
\hline
 \end{tabular}
 \end{center}
\caption{Maximal Property-P Subgroups of $S_{n}$}
\label{fig:max sn}
\end{table}

\section{Connected Subgroups and the Euler Transform}\label{sec:conn}
The conjugacy classes of subgroups of the symmetric group play an important role in the theory of combinatorial species as described in \cite{speciespaper}. 
Permutation groups have been used to answer many questions about species. Every species is the sum of its molecular subspecies. These molecular species correspond to 
conjugacy classes of subgroups of $\Sym{(n)}$. Molecular species decompose as products of atomic species which in turn correspond to connected subgroups of $\Sym{(n)}$ in the 
following sense. It will be convenient to denote the symmetric group on a finite set $X$ by $\Sym{(X)}$.

\begin{defn}\label{connected}
 For each $H \leq \mathsf{Sym}(X)$ there is a finest partition of $X = \bigsqcup Y_{i}$ such that $H = \prod H_{i}$ with $H_{i} \leq \Sym{(Y_{i})}$. We 
allow $H_{i} = 1$ when $|Y_{i}| = 1$. We say that $H$ is a connected subgroup of $\Sym{(X)}$ if the finest partition is $X$.
\end{defn}
\begin{eg}
Let $X = \{ 1,2,3,4 \}$ and consider $H = \langle (1,2), (3,4) \rangle$ and $H' = \langle (1,2)(3,4) \rangle$. Partitioning $X$ into $Y_{1} = \{ 1,2 \}, Y_{2} = \{ 3.4 \}$ 
gives $H = H_{1} \times H_{2}$ where $H_{1} =  \langle (1,2) \rangle \leq \mathsf{Sym}(\{ 1,2 \}) , H_{2} = \langle (3,4) \rangle \leq \mathsf{Sym}(\{ 3,4 \})$, hence $H$ is 
not connected. On the other hand, $H'$ is connected since there is no finer partition of $X$ which permits us to write $H'$ as a product of connected $H_{i}$.
\end{eg}

\noindent An algorithm to test a group $H$ acting on a set $X$ for connectedness checks each non-trivial $H$-stable subset $Y$ of $X$. If $H$ is the direct product of 
its action on $Y$ and its action on $X \setminus Y$ then $H$ is not connected.
 
 In general a subgroup $H \leq \Sym{(X)}$ is a product of connected subgroups $H_{i} \leq \Sym{(Y_{i})}$. Sequence \seqnum{A000638} records the number of molecular species of 
degree $n$ or equivalently the number of  
conjugacy classes of subgroups of $\Sym{(n)}$. Sequence  \seqnum{A005226} records the number of atomic species of degree $n$ or equivalently the number of  conjugacy classes of 
connected subgroups of $\Sym{(n)}$. These sequences are related by the Euler Transform.

\subsection{The Euler Transform}
If two sequences of integers $\{ c_{k} \} = ( c_{1}, c_{2}, c_{3}, \ldots )$ and  $\{ m_{n} \}  = ( m_{1}, m_{2}, m_{3}, \ldots )$
are related by
\begin{align}\label{eulertransform}
 1 + \sum \limits_{n \geq 1} m_{n}x^{n} = \prod \limits_{k \geq 1} \left(\frac{1}{1 - x^{k}} \right)^{c_{k}} .
\end{align}
Then we say that $\{ m_{n} \}$ is the Euler transform of $\{ c_{k} \}$ and that $\{ c_{k} \}$ is the inverse Euler transform of $\{ m_{n} \}$ (see \cite{BernSloane}). One sequence can 
be computed from the other by introducing the intermediate sequence $\{ b_{n} \}$ defined by
\begin{align}
 b_{n} = \sum \limits_{d | n}^{} d c_{d} = n m_{n} - \sum \limits_{k=1}^{n-1} b_{k} m_{n-k}.
\end{align}
Then 
\begin{align}
 m_{n} &= \frac{1}{n} \Bigl(    b_{n} + \sum \limits_{k =1}^{n-1} b_{k} m_{n-k}    \Bigr),& c_{n} =  \frac{1}{n} \sum \limits_{d | n}^{} \mu (n/d) b_{d},
\end{align}
where $\mu$ is the number-theoretic M\"obius function.

There 
are many applications of this pair of transforms (see \cite{sloane}). For example, the inverse Euler transform 
applied to the sequence of numbers of unlabeled graphs on $n$ nodes (\seqnum{A000088}) yields the sequence of numbers of connected graphs on $n$ nodes (\seqnum{A001349}). 
To understand how Formula \ref{eulertransform} can be used to count connected graphs we note that 
the coefficient of $x^n$ 
in the expansion of the product on the right hand side of Formula \ref{eulertransform} is
\begin{align}\label{multiset}
 m_{n} = \sum \limits_{1^{a_{1}}, 2^{a_{2}}, \ldots, n^{a_{n}} \vdash n}^{} \prod \limits_{i}^{}  \left(\! \!{c_{i} \choose a_{i}}\!\! \right) 
\end{align}
where $\bigl(\! \!{c_{i} \choose a_{i}}\!\! \bigr)$ denotes the number of $a_{i}$-element multisets chosen from a set of $c_{i}$ objects. On the other hand, an unlabeled 
graph on $n$ nodes, as a collection of connected components, can be characterized by a pair $(\lambda, (C_{1}, \ldots, C_{n}))$ where 
$\lambda =   1^{a_{1}}, 2^{a_{2}}, \ldots, n^{a_{n}}$ is a partition of $n$ and $C_{i}$ is a multiset of $a_{i}$ connected unlabeled graphs on $i$ nodes, 
for $1 \leq i \leq n$. If $c_{i}$ is the number of connected unlabeled graphs on $i$ nodes then, by Formula \ref{multiset}, $m_{n}$ is the total number of unlabeled 
graphs on $n$ nodes.

In the same way, the inverse Euler 
transform of \seqnum{A000638} (the number of conjugacy classes of subgroups of $S_{n}$) is \seqnum{A005226}, (the number of connected conjugacy classes of 
subgroups of $S_{n}$) as formalized in the following Lemma. 
\begin{Lemma}\label{subgroups and pairs}
There is a bijection between the conjugacy classes of subgroups of $S_{n}$ and the set of pairs of the form $( \lambda, (C_{1}, \ldots, C_{n}) )$ where 
$\lambda = 1^{a_{1}}, 2^{a_{2}}, \ldots, n^{a_{n}}$ is a partition of $n$ and $C_{i}$ is a multiset of $a_{i}$ conjugacy classes of connected subgroups of $S_{i}$ for
$ i = 1, \ldots, n$.
\end{Lemma}
\begin{proof}
Given a representative $H$ of the conjugacy class of subgroups $[H] \in \mathsf{Sub}(S_{n})/S_{n}$ we associate a pair $( \lambda, (C_{1}, \ldots, C_{n}) )$ 
to $H$ as follows. Write $H = \prod H_{k}$ where each $H_{k}$ is a connected subgroup of $\Sym{(Y_{k})}$. Then $X = \{ 1, \ldots, n \} = \bigsqcup Y_{k}$. Recording the size of each $Y_{k}$ yields a 
partition $\lambda =  1^{a_{1}}, 2^{a_{2}}, \ldots, n^{a_{n}}$. For $1 \leq i \leq n, C_{i}$ is the multiset of $S_{i}$-classes of subgroups $H_{k}$ with $|Y_{k}| = i$. 
Bijectivity follows from the fact that conjugate subgroups yield the same $\lambda$ and since $H^{g} = \prod H_{k}^{g}$, conjugate subgroups yield conjugate $C_{i}$.
\end{proof}

\subsection{Counting Connected Subgroups of the Alternating Group}
In Section \ref{sec:conn} we noted that molecular species correspond to conjugacy classes of subgroups of $\Sym{(n)}$ and that atomic species correspond to 
conjugacy classes of connected subgroups of $\Sym{(n)}$ in the sense of Definition \ref{connected}. In this Section we will count connected conjugacy classes of subgroups 
of the alternating group, up to $S_{n}$ conjugacy and $A_{n}$ conjugacy.

\subsubsection{$S_{n}$-Orbits of Subgroups of the Alternating Group}
 In order to count the number of $S_{n}$-conjugacy classes of subgroups of the alternating group we 
introduce the following notation. Let \[ \mathcal{B} = \{ H \leq S_{n} : H \leq A_{n} \mbox{ and } \mathcal{R} = \{ H \leq S_{n} : H \nleq A_{n} \}. \]
Then $\mathsf{Sub}(S_{n})/S_{n} = \mathcal{B}/S_{n} \sqcup \mathcal{R}/S_{n}$ and $\mathcal{B}/S_{n}$ is the set of $S_{n}$ conjugacy classes of subgroups of the alternating group. 
The set $\mathcal{R}/S_{n}$ is the set of conjugacy classes of subgroups of $S_{n}$ which are not contained in $A_{n}$. Table 
\ref{bluered} illustrates both of these sequences together with the numbers of conjugacy classes of subgroups of $S_{n}$ and $A_{n}$. In 
order to count the number of connected $S_{n}$-conjugacy classes of subgroups of $A_{n}$ we apply the inverse Euler transform 
to the sequence $|\mathcal{B}/S_{n}|$ in Table \ref{bluered}, to obtain

\[ \seqnum{A218968}: \mbox{    } 1 , 0 , 1 , 3 , 4 , 12 , 12 , 65 , 58 , 167 , 198 , 1207 , 1178. \]
We can also count the number of connected conjugacy classes of subgroups of $S_{n}$ not contained in $A_{n}$ (i.e. corresponding to  $\mathcal{R}/S_{n}$) by subtracting the sequence above from \seqnum{A005226} to obtain
\[ \seqnum{A218969}: \mbox{    }  0, 1, 1, 3, 2, 15, 8, 65, 66, 431, 443, 3643, 3594. \]
\begin{table}[H]
\begin{center}
\begin{tabular}{|l|r|r|r|r|}
\hline
 &\seqnum{A000638} &\seqnum{A029726} &\seqnum{A218966} &\seqnum{A218965} \\
\hline
$n$ & $|\mathsf{Sub}(S_{n})/S_{n}|$ & $|\mathsf{Sub}(A_{n})/A_{n}|$ & $|\mathcal{B}/S_{n}|$& $|\mathcal{R}/S_{n}|$ \\
\hline
$1$ &$1$ & $1$ & $1$&$0$  \\
$2$ &$2$ & $1$ & $1$&$1$  \\
$3$ &$4$ & $2$ & $2$&$2$  \\
$4$ &$11$ & $5$ & $5$&$6$  \\
$5$ &$19$ & $9$ & $9$& $10$ \\
$6$ &$56$ & $22$ & $22$&$34$  \\
$7$ &$96$ & $40$ & $37$&$59$  \\
$8$ &$296$ & $137$ & $112$&$184$  \\
$9$ &$554$ & $223$ & $195$&$359$  \\
$10$ &$1593$ & $430$ & $423$&$1170$  \\
$11$ &$3094$ & $788$ & $780$&$2314$  \\
$12$ &$10723$ & $2537$ & $2401$&$8322$  \\
$13$ &$20832$ & $4558$ & $4409$&$16423$  \\
\hline
\end{tabular}
\caption{Red and Blue Subgroups of $S_{n}$}\label{tab:redblue subgroups Sn}
\label{bluered}
\end{center}
\end{table}

\subsubsection{Connected Subgroups of the Alternating Group}
Every subgroup of $A_{n}$ is either connected or not connected with respect to the set $\{ 1, \ldots, n \}$ and shares this property with all subgroups in its 
$A_{n}$-conjugacy class. So we wish to count the number of $A_{n}$-conjugacy classes of connected subgroups of $A_{n}$. Unfortunately, the Euler transform does not apply 
to $A_{n}$-orbits. We test for connectedness a list of representatives of $\mathsf{Sub}(A_{n})/A_{n}$ in \textsf{GAP} and obtain
 \[\seqnum{A218967}: \mbox{     } 1, 0, 1, 3, 4, 12, 15, 87, 64, 168, 205, 1336, 1198. \]
\begin{remark}
There is a sequence in the encyclopedia, \seqnum{A116653}, which currently claims to count both the number of atomic species based on 
conjugacy classes of subgroups of the alternating group (i.e. the number of $S_{n}$-conjugacy classes of connected subgroups of 
$A_{n}$) and the 
number of $A_{n}$-conjugacy  classes of connected  subgroups of $A_{n}$. However this sequence is merely the inverse Euler transform of sequence \seqnum{A029726}, 
the number of conjugacy classes of subgroups of 
the alternating group. 
The number of $S_{n}$-conjugacy classes of connected subgroups of $A_{n}$ is sequence \seqnum{A218968} and the number of $A_{n}$-conjugacy classes of connected 
subgroups of $A_{n}$ is \seqnum{A216967}.
\end{remark}
\subsection{Connected Subgroups with Additional Properties}
Appealing to Definition \ref{connected} we can count the connected subgroups of $S_{n}$ which possess additional group theoretic properties. 
If the property of interest is compatible with taking direct products we can apply the inverse Euler transform to the sequence of numbers of all conjugacy classes of subgroups of 
$S_{n}$ with this property to obtain the sequence of numbers of conjugacy classes of connected subgroups of $S_{n}$ with this property.
Table  \ref{tab:connected subgroups Sn} records the number of connected subgroups of $S_{n}$
which additionally possess the properties listed in the first row of the table. 
Each of the sequences in Table \ref{tab:connected subgroups Sn} is the inverse Euler transform of the corresponding sequence in Table \ref{tab:sequences in Sn}.

\begin{table}[H]
\begin{center}
  \arraycolsep12pt
  $\begin{array}{|r|r|r|r|r|r|} 
\hline
& \seqnum{A000638}
& \seqnum{A218971}
& \seqnum{A218972}
& \seqnum{A218973}
& \seqnum{A218974}
\\
\hline
n& \text{$|\mathrm{Sub}(S_{n})/S_{n}|$} 
& \text{Abelian}
& \text{Nilpotent}
& \text{Solvable}
& \llap{\text{SupSol}}
\\
\hline
 1  & 1          &     1 &      1       &       1       &       1\\           
 2  & 2          &     1 &      1       &       1       &       1\\           
 3  & 4          &     1 &      1       &       2       &       2\\          
 4  & 11         &     3 &      4       &       6       &       4\\         
 5  & 19         &     1 &      1       &       4       &       4\\        
 6  & 56         &     6 &      9       &       23      &       15\\       
 7  & 96         &     1 &      1       &       16      &       13\\      
 8  & 296        &    17 &      69      &       122     &       81\\     
 9  & 554        &     5 &      8       &       109     &       77\\    
 10 & 1593       &    40 &      238     &       551     &       352\\  
 11 & 3094       &     2 &      2       &       570     &       406\\
 12 & 10723      &   162 &      2339    &       4633    &       2995\\
 13 & 20832      &     5 &      8       &       4224    &       2866\\
\hline
  \end{array}$
  
\caption{Connected Subgroups of $S_{n}$}\label{tab:connected subgroups Sn}
\end{center}
\end{table}

\subsection{Connected Partitions}
The number of conjugacy classes of cyclic subgroups of $S_{n}$ equals the number of partitions of $n$. The inverse Euler transform of this sequence yields the all ones 
sequence. This does not count the number of conjugacy classes of connected cyclic subgroups of $S_{n}$ since the direct product of cyclic groups is not necessarily cyclic.
\bigskip
\begin{defn}
 Let $\lambda = [x_{1}, \ldots, x_l ] $ be a partition of $n$. Let $G_{\lambda}$ be the simple graph with $l$ vertices labeled by $x_1, \ldots, x_l$ where two vertices are 
connected by an edge if and only if their labels $x_i, x_j$ are not coprime. We call the partition $\lambda$ connected if the graph $G_{\lambda}$ is connected.
\end{defn}

\begin{prop}\label{conpart}
 The number of conjugacy classes of connected cyclic subgroups of $S_{n}$ equals the number of connected partitions of $n$.
\end{prop}
\begin{proof}
Let $C = \langle c \rangle \leq S_{n}$ be cyclic. Then the cycle lengths of the permutation $c$ form a partition $\lambda$ of $n$.  
For simplicity, we assume $\lambda = [ x_1, x_2]$. Then $c = ab$ where $a$ is an $x_1$ cycle and $b$ is an $x_2$ cycle.  
Let $d = \gcd(x_1 , x_2)$. If $d = 1$ then $G_{\lambda}$ 
is not connected and   $ 1 = y x_1 + z x_2$, for some $y, z \in \mathbb{Z}$. Then $c^{yx_1} = a, c^{zx_2} = b$ and $C = \langle a \rangle \times \langle b \rangle$ is not connected. 
If $d > 1$ then $C$ does not contain a generator of $\langle a \rangle$ or of $\langle b \rangle$. The case for general $\lambda$ follows by a similar argument.
\end{proof}

\begin{eg}
There are $3$ connected partitions $\lambda$ of $13$. Their graphs $G_{\lambda}$ are shown in Figure \ref{fig:partitions of 13}.
\bigskip

 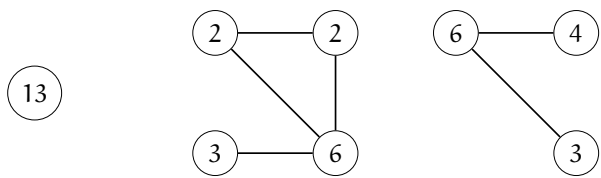
\begin{figure}[H]
\begin{center}
\scalebox{0.8}{
\begin{tikzpicture}

\node (13) at (2,3)[black,draw,circle]  {$13$};

\node (a21) at (5,4)[black,draw,circle]  {$2$};
\node (a22) at (7,4)[black,draw,circle]  {$2$};
\node (a3) at (5,2)[black,draw,circle]  {$3$};
\node (a6) at (7,2)[black,draw,circle]  {$6$};

\node (b6) at (9,4)[black,draw,circle]  {$6$};
\node (b4) at (11,4)[black,draw,circle]  {$4$};
\node (b3) at (11,2)[black,draw,circle]  {$3$};

\path
(a21) edge[thick] (a22)
(a21) edge[thick] (a6)
(a22) edge[thick] (a6)
(a3) edge[thick] (a6)

(b6) edge[thick] (b3)
(b6) edge[thick] (b4)

;
\end{tikzpicture}}
\end{center}
\caption{The Graphs of the Connected Partitions of $13$}
\label{fig:partitions of 13}
 \end{figure}

\end{eg}
Using Proposition \ref{conpart} we obtain the sequence of numbers of conjugacy classes of connected cyclic subgroups of $S_{n}$
\[ \seqnum{A218970}: \mbox{    }  1, 1, 1, 2, 1, 4, 1, 5, 3, 8, 2, 14, 3. \]

\begin{remark}
 There are two sequences in the encyclopedia which are quite similar to this sequence.  Sequence \seqnum{A018783} counts 
the number of partitions of $n$ into parts all of which have a common factor greater than $1$.  Sequence \seqnum{A200976} counts the number of partitions of $n$ 
such that each pair of parts (if any) has a common factor greater than $1$.
For $n \leq 13$, sequence \seqnum{A218970} above differs from both of these sequences 
when $n = 1, 11, 13$. 
\end{remark}


\section{Concluding Remarks} The sequences presented in this article have been computed using \textsf{GAP}. A \textsf{GAP} file containing the programs can be 
found at \url{www.maths.nuigalway.ie/~liam/CountingSubgroups.g}. The \textsf{GAP} table of marks library \textsf{Tomlib} can be found here \cite{tomlib} and is a requirement for computing many of the sequences 
presented. It is worth pointing out that Holt has determined all conjugacy classes of subgroups of $S_{n}$ for values of $n$ up to and including $n = 18$, (see \cite{holt}). 
The majority of the sequences presented in this article rely on the availability of the table of marks of $S_{n}$ and so we restrict our attention to $n \leq 13$. We are grateful to Des MacHale for suggesting many of 
the sequences that we compute  in this article. We thank the anonymous referees for helpful suggestions.

\vspace{9cm}

\appendix

\newcommand{\appsection}[1]{\let\oldthesection\thesection
  \renewcommand{\thesection}{Appendix \oldthesection}
  \section{#1}\let\thesection\oldthesection}

\newpage

\appsection{Additional Sequences}\label{appendix}
Using the methods described in this article the following additional sequences have been computed.
  \begin{center}
\begin{table} [H]
  \arraycolsep12pt
  $\begin{array}{|r|r|r|r|r|r|r|} 
\hline
&\seqnum{A029726}&\seqnum{A218934} &\seqnum{A218935} &\seqnum{A218936} &\seqnum{A218937} & \seqnum{A218938} \\
\hline
n& \text{$|\mathrm{Sub}(A_{n})/A_{n}|$} 
& \text{Abelian}
& \text{Cyclic}
& \text{Nilpotent}
& \text{Solvable}
& \llap{\text{SupSol}}
\\
\hline
 1  & 1          &     1 &       1  &   1       &       1       &       1\\           
 2  & 1          &     1 &       1  &   1       &       1       &       1\\           
 3  & 2          &     2 &       2  &   2       &       2       &       2\\          
 4  & 5          &     4 &       3  &   4       &       5       &       4\\         
 5  & 9          &     5 &       4  &   5       &       8       &       7\\        
 6  & 22         &     9 &       6  &   10      &       19      &       14\\       
 7  & 40         &    12 &       8  &   13      &       33      &       22\\      
 8  & 137        &    30 &      12  &   53      &       122     &       70\\     
 9  & 223        &    41 &      17  &   69      &       192     &       122\\    
 10 & 430        &    60 &      23  &   122     &       364     &       225\\  
 11 & 788        &    81 &      29  &   160     &       650     &       395\\
 12 & 2537       &   193 &      40  &   734     &       2194    &       1240\\
 13 & 4558       &   243 &      52  &   848     &       3845    &       2185\\
\hline
\end{array}$
\caption{Conjugacy classes of subgroups of $A_{n}$}\label{tab:sequences in An}
\end{table}
\end{center}

  \begin{center}
\begin{table} [H]
  \arraycolsep3pt
  $\begin{array}{|r|r|r|r|r|r|r|} \hline
 &\seqnum{A005432} &\seqnum{A062297} &\seqnum{A051625} &\seqnum{A218939} &\seqnum{A218940} &\seqnum{A218941}  \\ \hline
 n &\text{$|\mathrm{Sub}(S_{n})|$} 
& \text{Abelian}
& \text{Cyclic}
& \text{Nilpotent}
& \text{Solvable}
& \llap{\text{SupSol}}
\\
\hline
1	   &    1  	&    1   &    1   &    1   &    1   &    1 \\ 
2	   &    2   	&    2   &    2   &    2   &    2   &    2 \\   
3	   &    6   	&    5   &    5   &    5   &    6   &    6 \\
4	   &    30   	&    21   &    17   &    24   &    30   &    28 \\
5	   &    156   	&    87   &    67   &    102   &    154   &    144 \\
6	   &    1455   	&    612   &    362   &    837   &    1429   &    1259 \\ 
7	   &    11300   &    3649   &    2039   &    5119   &    11065   &    9560 \\
8	   &    151221   &    35515   &    14170   &    78670   &    148817   &    123102 \\
9	   &    1694723   &    289927   &    109694   &    664658   &    1667697   &    1371022 \\ 
10 	   &    29594446   &    3771118   &    976412   &    13514453   &    29103894   &    23449585 \\
11 	   &    404126228   &    36947363   &    8921002   &    137227213   &    396571224   &    317178020 \\
12 	   &    10594925360   &    657510251   &    101134244   &    4919721831   &    10450152905   &    8296640115 \\
13 	   &    175238308453   &    7736272845   &    1104940280   &    60598902665   &    172658168937   &    136245390535 \\
\hline
\end{array}$
\caption{Total number of subgroups of $S_{n}$}\label{tab:totsequences in Sn}
\end{table}
\end{center}

\begin{table}[H]
 \begin{center}
\begin{tabular}{|l|r|r|r|r|r|}
\hline 
 & \seqnum{A218955} & \seqnum{A218956} & \seqnum{A218957} & \seqnum{A218958} & \seqnum{A218959}\\
\hline
n & Solvable & SupSol & Abelian & Cyclic & Nilpotent\\
\hline
    1& 1& 1& 1& 1& 1 \\    2& 1& 1& 1& 1& 1 \\    3& 1& 1& 4& 4& 4 \\    4& 1& 7& 11& 13& 7 \\    5& 21& 31& 51& 31& 31 \\    6& 76& 101& 241& 246& 211 \\  
    7& 456& 491& 1506& 1296& 1156 \\    8& 1956& 3011& 9649& 10774& 5419 \\    9& 12136& 18467& 80281& 83238& 40027 \\    10& 80836& 114983& 640741& 788820& 348331 \\  
    11& 807676& 1283723& 6196576& 6835170& 3204796 \\    12& 8779816& 13380643& 66883411& 81364944& 38422891 \\    13& 104127596& 148321603& 775421219& 848378532& 467645179\\
\hline
 \end{tabular}
 \end{center}
\caption{Total number of maximal property-P subgroups of $S_{n}$}
\label{fig:tot max Sn}
\end{table}

  \begin{center}
\begin{table} [H]
  \arraycolsep5pt
  $\begin{array}{|r|r|r|r|r|r|r|r|} \hline
 & \seqnum{A029725} &\seqnum{A218942} & \seqnum{A051636} &\seqnum{A218943} &\seqnum{A218944} &\seqnum{A218945} \\
\hline
n& \text{$|\mathrm{Sub}(A_{n})|$} 
& \text{Abelian}
& \text{Cyclic}
& \text{Nilpotent}
& \text{Solvable}
& \llap{\text{SupSol}}
\\
\hline
 1  & 1          &     1 &       1  &   1       &       1       &       1\\           
 2  & 1          &     1 &       1  &   1       &       1       &       1\\           
 3  & 2          &     2 &       2  &   2       &       2       &       2\\          
 4  & 10          &     9 &       8  &   9       &       10       &       9\\         
 5  & 59          &     37 &       32  &   37       &       58       &       53\\        
 6  & 501         &     207 &       167  &   252      &       488      &       418\\       
 7  & 3786         &    1192 &       947  &   1507      &       3664      &       2894\\      
 8  & 48337        &    11449 &      6974  &   21739      &       47210     &      33675 \\     
 9  & 508402        &    93673 &      53426  &   186983      &        498102    &      369763 \\    
 10 & 6469142       &    892783 &      454682  &    2369258    &       6293475     &       4769542\\  
 11 & 81711572        &    8534308 &      4303532  &   22872863    &       78805290     &       58853842\\
 12 & 2019160542       &   148561283 &      50366912  &   746597568     &       1960342409    &      1395051100 \\
 13 & 31945830446      &   1740198891 &      553031624  &   9157758326     &       31130243721    &      21847262156 \\
\hline
  \end{array}$
\caption{Total no of subgroups of $A_{n}$}\label{tab:totsequences in An}
\end{table}
\end{center}

\begin{table}[H]
 \begin{center}
\begin{tabular}{|l|r|r|r|r|r|}
\hline 
 & \seqnum{A218946}  & \seqnum{A218947}  & \seqnum{A218948}  & \seqnum{A218949}  & \seqnum{A218950} \\
\hline 
n & Solvable & SupSol & Abelian & Cyclic & Nilpotent\\
\hline
1 & 1 & 1 & 1 & 1 & 1\\
2 & 1 & 1 & 1 & 1 & 1\\
3 & 1 & 1 & 1 & 1 & 1\\
4 & 1 & 2 & 2 & 2 & 2\\
5 & 3 & 3 & 3 & 3 & 3\\
6 & 4 & 3 & 5 & 4 & 3\\
7 & 5 & 4 & 6 & 5 & 5\\
8 & 6 & 6 & 13 & 6 & 6\\
9 & 10 & 8 & 19 & 8 & 7\\
10 & 12 & 10 & 22 & 10 & 9\\
11 & 14 & 13 & 27 & 14 & 12\\
12 & 17 & 18 & 40 & 20 & 17\\
13 & 24 & 22 & 54 & 24 & 20\\
\hline
 \end{tabular}
 \end{center}
\caption{Maximal property-P subgroups of $A_{n}$}
\label{fig:max an}
\end{table}

\begin{table}[H]
\begin{center}
  \arraycolsep12pt
  $\begin{array}{|r|r|r|r|r|r|} 
\hline 
 & \seqnum{A029726}  & \seqnum{A218951}  & \seqnum{A218952}  & \seqnum{A218953}  & \seqnum{A218954} \\
\hline
n& \text{$|\mathrm{Sub}(A_{n})/A_{n}|$} 
& \text{Abelian}
& \text{Nilpotent}
& \text{Solvable}
& \llap{\text{SupSol}}
\\
\hline
 1  & 1          &     1 &      1       &       1       &       1\\           
 2  & 1          &     0 &      0       &       0       &       0\\           
 3  & 2          &     1 &      1       &       1       &       1\\          
 4  & 5          &     2 &      2       &       3       &       2\\         
 5  & 9          &     1 &      1       &       3       &       3\\        
 6  & 22         &     3 &      4       &       10      &       6\\       
 7  & 40         &     1 &      1       &       11      &       6\\      
 8  & 137        &    14 &      36      &       80     &       42\\     
 9  & 223        &     5 &      9       &       52     &       39\\    
 10 & 430        &    12 &      49      &       145     &       85\\  
 11 & 788        &     2 &      2       &       165     &       104\\
 12 & 2537       &   69 &      489      &       1208    &       686\\
 13 & 4558       &     3 &      4       &       1033    &       617\\
\hline
  \end{array}$
  
\caption{Connected subgroups of $A_{n}$}\label{tab:connected subgroups An}
\end{center}
\end{table}
\noindent The number of connected even partitions of $n$
\[\seqnum{A218975}: \mbox{    }   1, 0, 1, 1, 1, 2, 1, 3, 3, 4, 2, 8, 2. \]

\begin{table}[H]
 \begin{center}
\begin{tabular}{|l|r|r|r|r|r|}
\hline 
 & \seqnum{A218960} & \seqnum{A218961} & \seqnum{A218962} & \seqnum{A218963} & \seqnum{A218964} \\
\hline 
n & Solvable & SupSol & Abelian & Cyclic & Nilpotent\\
\hline
  1& 1& 1& 1& 1& 1 \\   2& 1& 1& 1& 1& 1 \\   3& 3& 3& 3& 3& 3 \\   4& 1& 10& 10& 9& 10 \\   5& 36& 40& 30& 30& 30 \\   6& 225& 110& 115& 100& 110 \\  
   7& 686& 645& 861& 665& 1001 \\   8& 4655& 5670& 10536& 3885& 4005 \\   9& 28728& 47754& 78474& 33093& 45696 \\   10& 397005& 311850& 1008000& 371700& 379155 \\  
   11& 2210890& 3014550& 9302964& 3790875& 4913040 \\   12& 26975025& 24022845& 73024380& 37839285& 36701280 \\   13& 26121667& 46950904& 563291872& 350984414& 158538380 \\
\hline
 \end{tabular}
 \end{center}
\caption{Total number of maximal property-P subgroups of $A_{n}$}
\label{fig:tot max P An}
\end{table}

\bibliographystyle{amsplain}
\bibliography{References}

\bigskip
\hrule
\bigskip

\noindent 2010 {\it Mathematics Subject Classification}:
Primary 20B40; Secondary 20D30, 19A22.

\noindent \emph{Keywords:}   
symmetric group,
alternating group,
table of marks,
subgroup pattern.

\bigskip
\hrule
\bigskip

\noindent (Concerned with sequences
 \seqnum{A218909}, \seqnum{A218910}, \seqnum{A218911}, \seqnum{A218912}, \seqnum{A218913}, \seqnum{A218914}, \seqnum{A218915}, \seqnum{A218916}, \seqnum{A218917}, \seqnum{A218918}, \seqnum{A218919}, \seqnum{A218920}, \seqnum{A218921}, \seqnum{A218922}, \seqnum{A218923}, \seqnum{A218924}, \seqnum{A218925}, \seqnum{A218926}, \seqnum{A218927}, \seqnum{A218928}, \seqnum{A218929}, 
  \seqnum{A218930}, \seqnum{A218931}, \seqnum{A218932}, \seqnum{A218933}, \seqnum{A218934}, \seqnum{A218935}, \seqnum{A218936}, \seqnum{A218937}, \seqnum{A218938}, \seqnum{A218939}, \seqnum{A218940}, \seqnum{A218941}, \seqnum{A218942}, \seqnum{A218943}, \seqnum{A218944}, \seqnum{A218945}, \seqnum{A218946}, \seqnum{A218947}, \seqnum{A218948}, \seqnum{A218949}, \seqnum{A218950}, 
  \seqnum{A218951}, \seqnum{A218952}, \seqnum{A218953}, \seqnum{A218954}, \seqnum{A218955}, \seqnum{A218956}, \seqnum{A218957}, \seqnum{A218958}, \seqnum{A218959}, \seqnum{A218960}, \seqnum{A218961}, \seqnum{A218962}, \seqnum{A218963}, \seqnum{A218964}, \seqnum{A218965}, \seqnum{A218966}, \seqnum{A218967}, \seqnum{A218968}, \seqnum{A218969}, \seqnum{A218970}, \seqnum{A218971}, 
  \seqnum{A218972}, \seqnum{A218973}, \seqnum{A218974}, \seqnum{A000638}, \seqnum{A029726}, \seqnum{A005432}, \seqnum{A029725}, \seqnum{A005226}, \seqnum{A116653},
\seqnum{A018783} and \seqnum{A200976}).

\bigskip
\hrule
\bigskip

\end{document}